\documentclass[twoside,11pt]{amsart} 
\usepackage{enumitem}
\usepackage{epsfig}
\usepackage{latexsym} 
\usepackage{amsmath} \usepackage{amssymb}
\usepackage{amsfonts,amssymb,amsmath,amscd,amsthm}
\usepackage{epsfig}
\usepackage[dvipsnames]{xcolor}
\usepackage{placeins}   

 \keywords{ primes, gaps, prime constellations, Eratosthenes sieve}

\subjclass{11N05, 11A41, 11A07}

\makeatother

\newtheorem{lemma}{Lemma}

\theoremstyle{definition}

\newcommand{\pgap}   {{\mathcal G}}

\newcommand{\lil}   {\scriptstyle }

\begin{document}

\title[On nonconvex constellations among primes I]{On nonconvex constellations among primes I }

\date{9 Apr 2026 - completed calculations in Table~\ref{GammaTbl}}

\author{Fred B. Holt}
\address{\tt fbholt@primegaps.info, www.github.com/fbholt/Primegaps-v2}
																														
\begin{abstract}
Extending our work on the $k$-tuple conjecture, we apply those methods to the Engelsma counterexamples (narrow constellations)
of length $J=459$ and span $|s|=3242$.
We track the evolution of these $58$ counterexamples from inadmissible driving terms starting in the cycle of gaps $\pgap(11^\#)$
up through their first appearance in $\pgap(113^\#)$.  We continue developing primorial coordinates for each admissible instance
through a breadth-first exhaustive search through $\pgap(211^\#)$, at which point we need to develop strategies for depth-first searches
for an instance that would survive Eratosthenes sieve.  Our calculations show that {\em none} of the $(459,3242)$-counterexamples
occur before $9.7\,E73$.

For each of the $58$ Engelsma $(459,3242)$-counterexamples we calculate its asymptotic relative population, among other constellations
of length $J=459$, and we study how these counterexamples work.
\end{abstract}

\maketitle
\pagestyle{myheadings}
\thispagestyle{empty}
\baselineskip=12.875pt
\vskip 30pt

\section{Introduction}
This is a specific extension of our recent work on the $k$-tuple conjecture \cite{FBHktuple}.  We have shown that all admissible constellations
arise and persist in the cycles of gaps $\pgap(p^\#)$, and that the population of {\em every} admissible constellation of length $J$ ultimately grows
as ${\Theta(\prod_{p > J+1}(p-J-1))}$.

We apply the tools and methods in that paper \cite{FBHktuple} to the class of nonconvex constellations.
An admissible constellation $s$ of length $J$ is said to be {\em nonconvex} iff 
$$ \pi(|s|) < J.$$
That is, the number of small primes in the interval $(0,|s|]$ is less than the number of gaps in this constellation.  For an occurrence of
$s$ among primes, starting at the prime $x$ we would have
\begin{eqnarray*}
\pi(x+|s|) & = & \pi([0,x]) \; + \; \pi((x,x+|s|]) \\
  & = & \pi(x) \; + \; J \\
  & > & \pi(x) \; + \; \pi(|s|)
\end{eqnarray*}

\begin{figure}[hbt]
\centering
\includegraphics[width=4in]{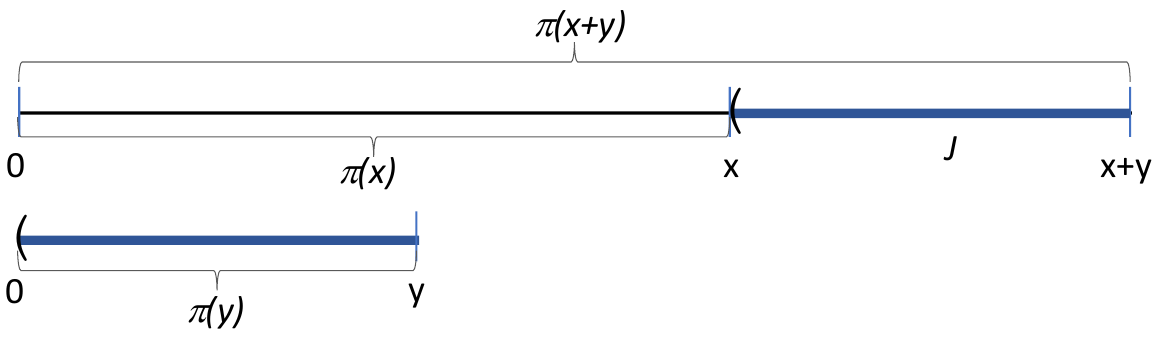}
\caption{\label{ConvexFig} Convexity and nonconvexity.  We are comparing the density of primes in an interval away from the origin, with the
density of an interval of equal length starting at the origin. }
\end{figure}

In 1974 Hensley and Richards \cite{HRich,Rich} showed that nonconvex constellations exist, and in 2005 Engelsma and others \cite{Engtbl, Sutherland} identified numerous examples.  No actual instance of a nonconvex constellation among primes has yet been discovered.
That is, nonconvex admissible constellations $s$ have been identified, but we do not yet know of a specific instance $\gamma_0$
where $s$ occurs among primes.

Here we study a few of the smaller nonconvex constellations identified by Engelsma.
\begin{itemize}
\item We propose a correction to the indexing and scoring of narrow constellations.
\item Under the revised scoring, the smallest nonconvex constellations are 
$${(J,|s|) = (458,3240) \; {\rm and} \; (459,3242)}$$
\item We explore the primorial coordinates for instances of the $(459,3242)$-counterexamples, seeking to confirm an instance that
arises among primes.
\end{itemize}

\subsection{How do these counterexamples work?}
In Figure~\ref{Eng4vpiFig} we plot the constellation among small primes against the constellation $s_4$ from the Engelsma
counterexamples $(459,3242)$.

\begin{figure}[hbt]
\centering
\includegraphics[width=5in]{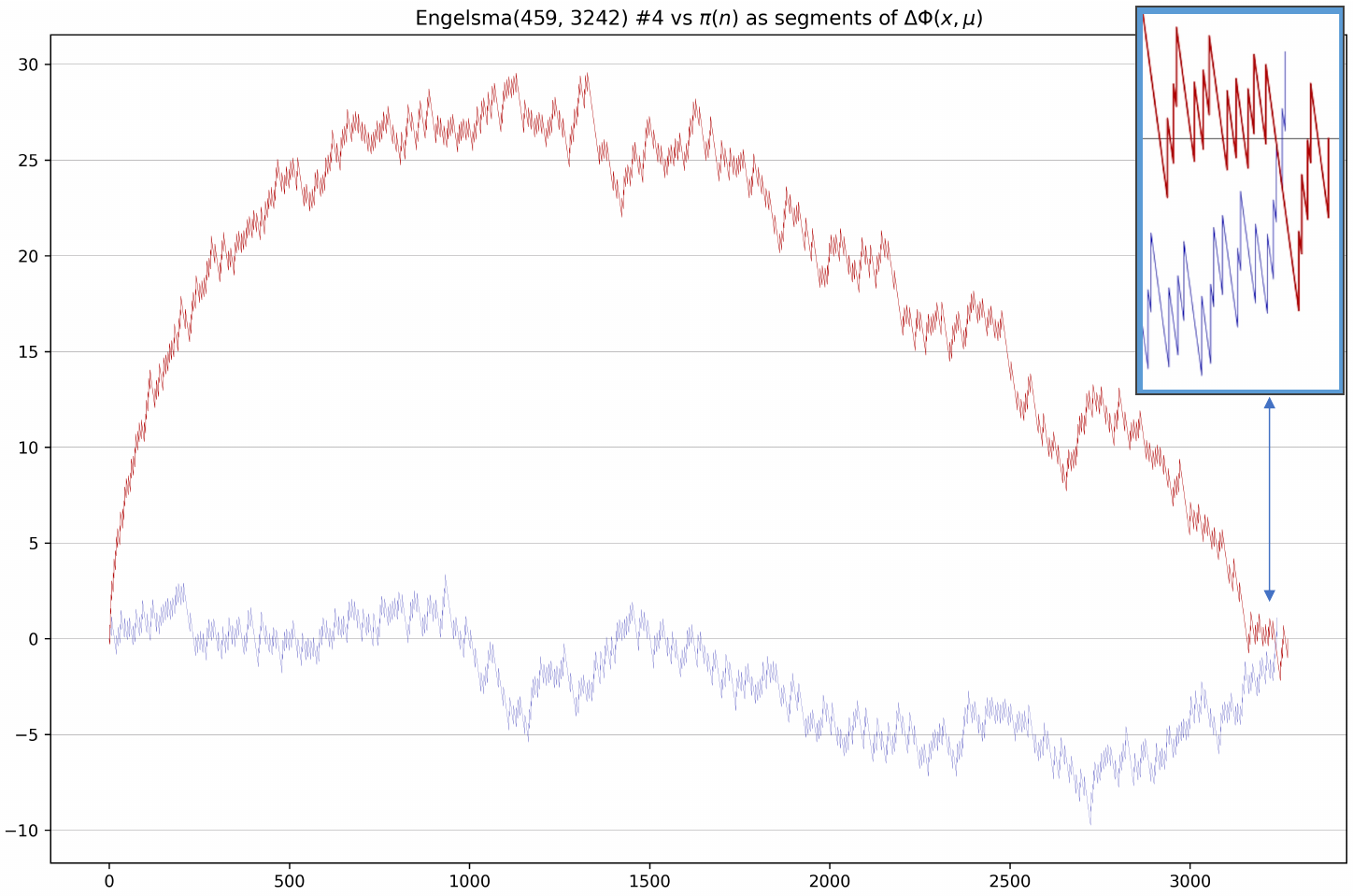}
\caption{\label{Eng4vpiFig} One constellation in $(J,|s|)=(459,3242)$ plotted (in blue) as a segment of $\Delta \Phi(x,\mu)$, compared to the
gaps between primes (in red) from $0$ to $n$. The small gaps in $\pi(x)$ create a huge early lead, but as larger gaps occur between primes,
it gives narrow constellations an opportunity to overtake it.  The inset shows where the graph for $s_4$ surpasses the sequence of prime gaps.
We see that $s_4$ first crosses above the red graph for $(458,3240)$, and stays above for $(459,3242)$.}
\end{figure}

For the graphs we use $\Delta \Phi$ from \cite{FBHprough}. 
$$ \Delta \Phi(x) \; = \; \Phi(x) -\frac{1}{\mu}x $$
where $\Phi(x)$ is the number of gaps in the constellation up to the span $x$.
Here we use the average $\mu = 7.087146$ for the first $459$ prime gaps.  The graph of $\Delta \Phi$ consists of vertical segments
of length $1$ where the next prime (or candidate prime) occurs, followed by downward sloping segments of slope $\frac{-1}{\mu}$.
The run of a downward-sloping segment is the size of the corresponding gap.

Note the role that large gaps play as narrow constellations chase the sequence of initial primes. 
The constellation for the small primes gets a substantial lead, and we can see why -- before the work of Hensley and Richards
\cite{HRich} and the polymath project \cite{Engtbl, Sutherland} -- the convexity conjecture seemed so intuitive.
The constellation of gaps between the $15$ primes in the interval $(0,47]$ is inadmissible and impossibly dense.  In Figure~\ref{Eng4vpiFig}
we see this aggressive lead in $\pi(n)$ continue up through a span of $1000$ or $1200$.

Then as larger gaps occur between the primes, there is an opportunity for admissible constellations to catch up.
Up to $p_{459}=3253$ the small primes have $14$ gaps of size $20$ or more.  See Table~\ref{Engghist}.

The gap of $22$ between
$p_{457}=3229$ and $p_{458}=3251$ finally enables the narrow admissible constellations to produce the first counterexamples.  Then the
jump of $34$ from $(459,3242)$ to $(460,3276)$ cedes this back to convex constellations until $(J,|s|)=(595,4352)$.  
It is interesting that the shortest gap
that could be prepended or appended to a $(459,3242)$-counterexample is $34$.
The unique instance of the constellation $s_4$ in $\pgap(137^\#)$ is preceded by a gap $70$ and followed by a gap of $34$.

\begin{table}[!h]
\centering
\setlength{\tabcolsep}{4pt}
\begin{tabular}{lr|cccccccccccccccccc|} 
 & \multicolumn{19}{c|}{Distribution of $459$ gaps} \\
 $g$ & & $\lil 1$ & $\lil 2$ & $\lil 4$ & $\lil 6$ & $\lil 8$ & $\lil 10$ & $\lil 12$ & $\lil 14$ & $\lil 16$ & $\lil 18$ & $\lil 20$ & $\lil 22$ & $\lil 24$ 
   & $\lil 26$ & $\lil 28$ & $\lil 30$ & $\lil 32$ & $\lil 34$ \\ \hline
$p_{459}$ & $3253$ & $1$ & $86$ & $92$ & $112$ & $44$ & $43$ & $32$ & $18$ & $8$ & $9$ & $3$ & $5$ & $2$ & $2$ & $1$ & & & $1$ \\
$s_4$ & $3242$ & & $85$ & $77$ & $118$ & $44$ & $57$ & $32$ & $18$ & $12$ & $11$ & $2$ & $2$ & $1$ & & &  & & \\ \hline
\end{tabular}
\caption{ \label{Engghist} Populations of the gaps between the primes from $0$ to $p_{459}=3253$, 
and in one $(459,3242)$-counterexample $s_4$.}
\end{table}

\section{Indexing and scoring narrow admissible $k$-tuples}
A narrow admissible $k$-tuple of span $y= \gamma_{k-1}-\gamma_0$ has a natural correspondence to an admissible constellation 
$s$ of length $J=k-1$ and span $|s|=y$.

Where an instance of the $k$-tuple or the constellation $s$ actually occur, the convexity constraint is
$$ \pi(\gamma_0+y) \le \pi(\gamma_0) + \pi(y)$$
A counterexample would contradict this inequality, yielding instead
\begin{eqnarray*}
 \pi(\gamma_0 + y)  & > & \pi(\gamma_0) + \pi(y) \\
 \pi(\gamma_0) + J & > & \pi(\gamma_0) + \pi(y) \\
 {\rm or} \hspace{0.5in} J & > & \pi(y)
 \end{eqnarray*}

We will index the counterexamples by length $J=k-1$ and span ${|s| = \gamma_J-\gamma_0}$, and we will score the nonconvexity of
a constellation by the difference $J-\pi(|s|)$.

Engelsma \cite{Engtbl} uses $k$ and $w=|s|+1$ as indices and calculates a nonconvexity score of $k-\pi(w)$.  
When $w$ is prime $\pi(w)=\pi(|s|)+1$, an occasional difference.  Using $k$ instead of $J$ makes the scores
too high by $1$, except when $w$ is prime.
Sutherland \cite{Sutherland} corrects $w$ to the span $|s|$ but still uses $k$.  So the inflated scores persist.

\begin{figure}[hbt]
\centering
\includegraphics[width=5in]{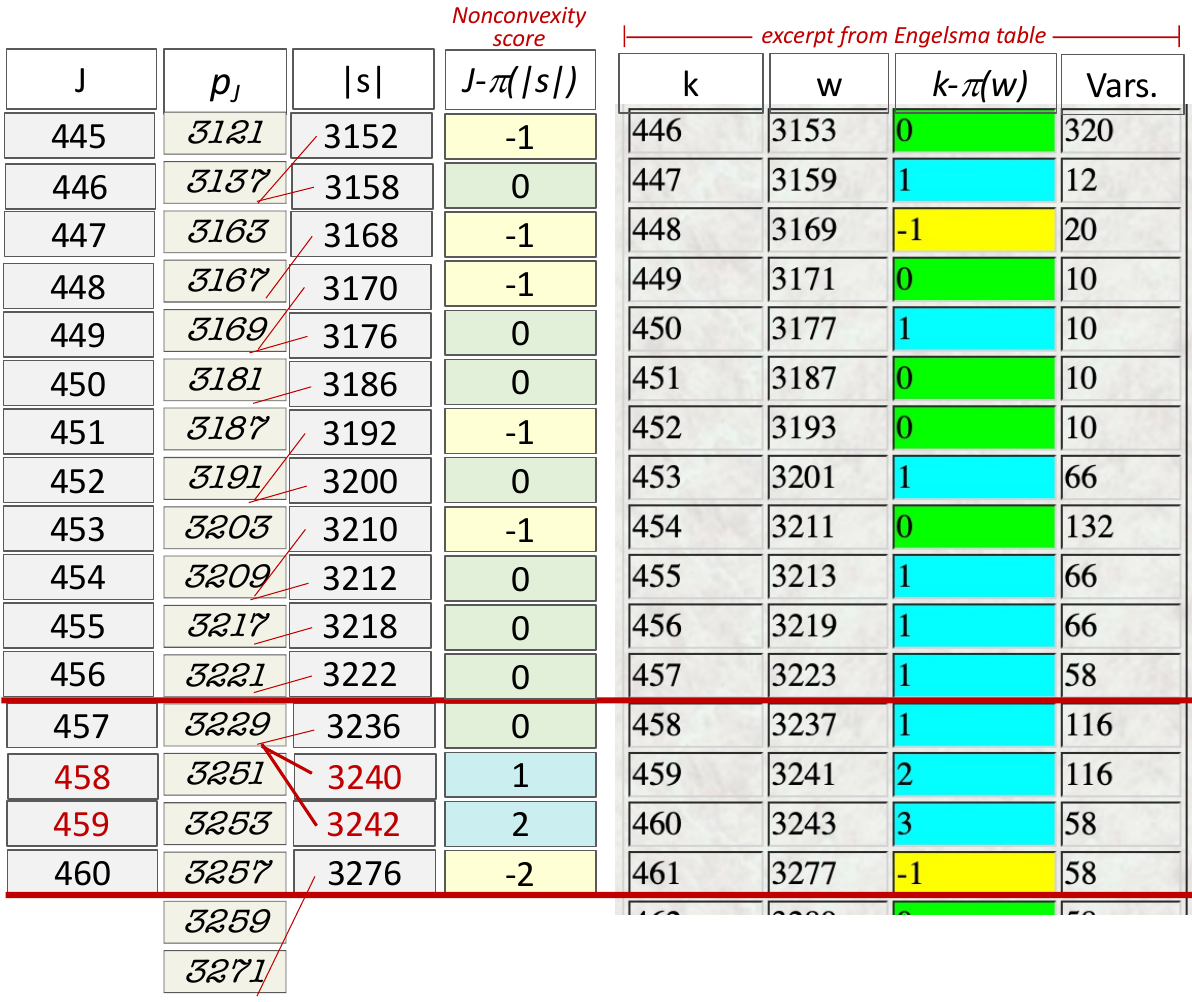}
\caption{\label{EngTblFig} A section of Engelsma's table, shown at right, supplemented with columns for $J = k-1$, span $|s|=w-1$, 
and the true nonconvexity score $J-\pi(|s|)$.  We list the primes $p_J$ in the second column for reference.  The red segments show
where in the sequence of primes the entries $|s|$ fall, illustrating $\pi(|s|)$.}
\end{figure}

The difficulty here is how the left end of the interval from $x$ to $x+y$ is treated.  For example, in Figure~\ref{EngTblFig} 
suppose a constellation ${(J,|s|)=(446,3158)}$ did occur, starting at $\gamma_0$.  This instance has score $k-\pi(w)=1$.
 $\gamma_0$ must be prime, or this is not really an occurrence of $s$.  The count of $x=\gamma_0$ as a prime
either happens in $[0,x]$ or in $[x,x+y]$ but not both.  So we either have
$$
\pi(x+y) = \pi(x-1) + k \hspace{0.2in} {\rm or} \hspace{.2in} \pi(x+y) = \pi(x) + J,
$$
which yield the same total, since $x$ is prime.

For the constellations with $(J,|s|)=(446,3158)$, their instances among primes meet but do not exceed the convexity bound.
\begin{eqnarray*}
\pi(x+3158) & \stackrel{?}{\le} & \pi(x) + \pi(3158) \\
 \pi(x) + 446 & \stackrel{\surd}{=} &  \pi(x) + 446
\end{eqnarray*}

If we were to count the left endpoint of the interval $[x,x+y]$ as a prime twice, 
once in $[0,x]$ and again in $[x,x+y]$, then every twin prime would violate the
convexity conjecture.  For twin primes the $2$-tuple is $[0,2]$.  Under the erroneous scoring, we have two primes covered by the tuple
of span ${|s|=2}$, but ${\pi(2)=1 < k= 2}$, and every twin prime would be a counterexample to the convexity conjecture.

Although the scores $k-\pi(w)$ in the summary tables \cite{Engtbl} are generally padded by $1$, we simply note the correction and make use 
of the excellent underlying data.  In Figure~\ref{EngTblFig} we record the nonconvexity score $J-\pi(|s|)$ in the fourth column.

\begin{lemma}
From the Engelsma data, the smallest counterexamples to the convexity conjecture would occur for $({J,|s|)=(458,3240)}$
and  ${(459,3242)}$.
\end{lemma}

See the section of Engelsma's table and supplementary columns in Figure~\ref{EngTblFig}.  
Wherever ${J-\pi(|s|) >0}$ we have admissible constellations
for which an instance $x=\gamma_0$ among primes would initiate a counterexample to the convexity conjecture.  For each $J$, we are looking for 
admissible constellations for which the span $|s| < p_J$; that is, the interval required to cover an additional $J$ primes is narrower than the
first $J$ primes themselves.

We call a constellation $s$ a {\em $(J,|s|)$-counterexample} if $s$ has length $J$ and  span $|s|$ such that 
$$ J > \pi(|s|) \hspace{.2in} {\rm or~equivalently} \hspace{.2in}  |s| < p_J.$$  
To be clear, the constellation $s$ is not itself a counterexample to the convexity conjecture, but any 
instance $\gamma_0$ of $s$ among primes would be.
We have the narrow constellation $s$, but we still need the initial generator $\gamma_0$ for an instance of $s$ among primes.

\section{Engelsma $(459,3242)$-counterexamples}

The last column in Figure~\ref{EngTblFig} lists the number of distinct constellations in that $(J, |s|)$ class.
Engelsma \cite{Engtbl, Sutherland} found $58$ distinct $(459,3242)$-counterexamples.  Sutherland 
has archived this data \cite{Sutherland}.  We study these $(459,3242)$-counterexamples here.

\subsection{Why length $J=459$?}  
Every one of the $58$ $(459,3242)$-counterexamples starts and ends with gaps $g=2$.  
If we drop the last gap $g=2$ from these, we get $58$ $(458,3240)$-counterexamples that begin with a gap $g=2$. 
If instead we drop the first gap $g=2$, we get $58$ $(458,3240)$-counterexamples that end with a gap $g=2$.
Together these are the $116$ $(458,3240)$-counterexamples identified by Engelsma.

We focus on the $(459,3242)$-counterexamples since these contain all of the $(458,3240)$ ones.
Additionally the nonconvexity score $J-\pi(|s|)$ for $(459,3242)$ is
$$459 - \pi(3242) \, = \, 459-457 \, = \, 2.$$

\subsection{Prefixes for the primorial coordinates}
Here we analyze a few characteristics of one of Engelsma's $(459,3242)$-counterexamples.  
In our GitHub repository\footnote{\tt https://github.com/fbholt/Primegaps-v2}
there is a notebook for repeating this analysis for any of these counterexamples.

From Engelsma's table (Figure~\ref{EngTblFig}) we see that there are $58$ different constellations among the
$(459,3242)$-counterexamples.  
When we focus on the details of a specific example, we will work with the $(459,3242)$-constellation of index $\#4$ in the
GitHub data.
We call this constellation $s_4$ and list its gaps in Figure~\ref{Eng459Fig}.
We list the constellation in Figure~\ref{Eng459Fig}, 
highlighted in Figure~\ref{EngPrefixFig}
and graphed versus the constellation among small primes in Figure~\ref{Eng4vpiFig}.

We start in the cycle $\pgap(11^\#)$.  Of the $58$ $(459,3242)$-counterexamples, $29$ share a single
inadmissible driving term in $\pgap(11^\#)$ with $\gamma_0=107$.  The other $29$ are the mirror images of the first $29$, 
and these mirror images share a unique inadmissible driving term in $\pgap(11^\#)$ with $\gamma_0=1271$.

\begin{figure}[hbt]
\centering
\includegraphics[width=5in]{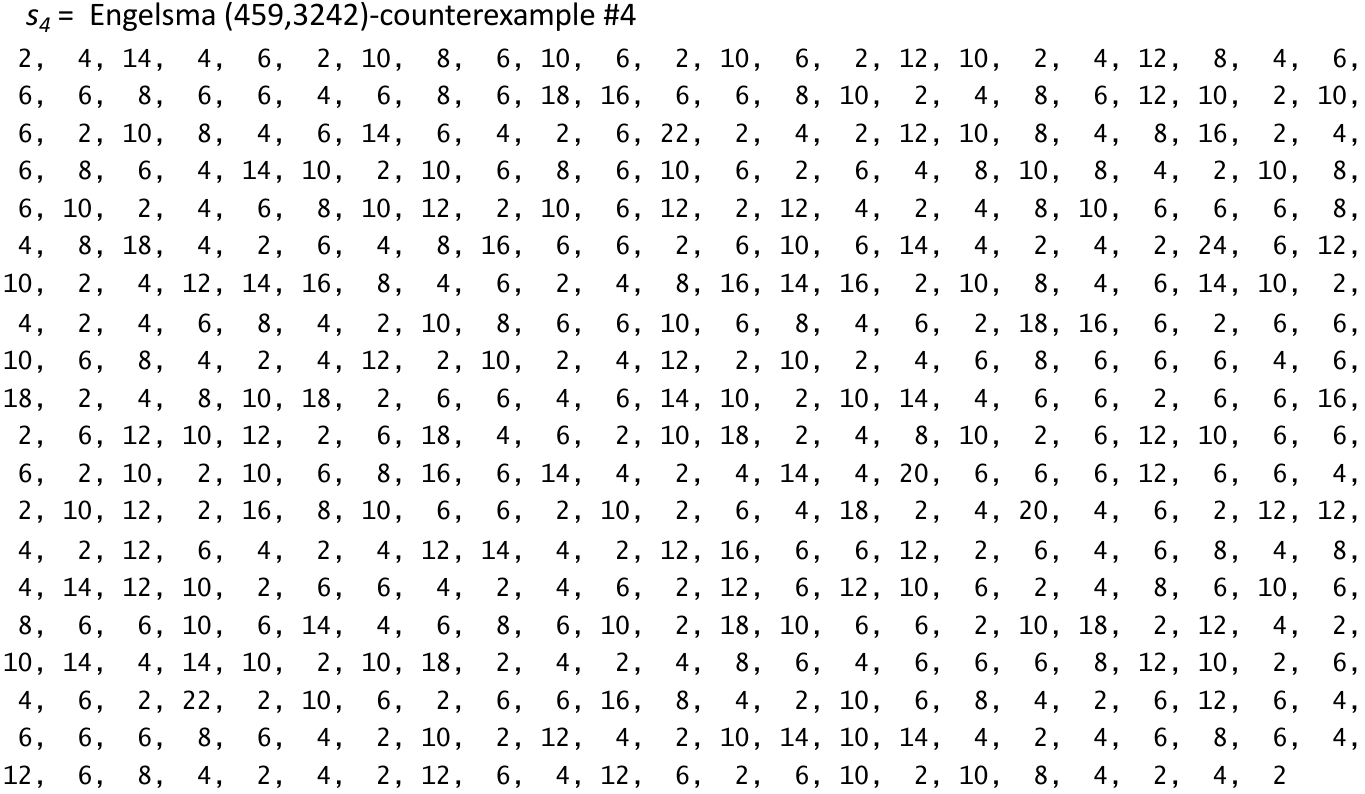}
\caption{\label{Eng459Fig} One example $s_4$ of the $58$ narrow constellations in $(J,|s|)=(459,3242)$.}
\end{figure}

Across several cycles $\pgap(p^\#)$ there continues to be a single driving term for each of these constellations.  We record these
unique prefixes by their primorial coordinates in Figure~\ref{EngPrefixFig}.
Figure~\ref{EngPrefixFig} tabulates the primorial coordinates for the unique prefixes for the first $29$ $(459,3242)$-counterexamples 
through the cycles $\pgap(p^\#)$.  The other $29$ $(459,3242)$-counterexamples are the mirror images of this first set.

For each of these $29$ $(459,3242)$-counterexamples, the constellation itself first occurs in $\pgap(113^\#)$.  
These counterexamples all have unique occurrences into $\pgap(131^\#)$.  Our example $s_4$ has a unique instance in 
$\pgap(137^\#)$ as well.
For $s_4$ its unique occurrence in $\pgap(137^\#)$ has 
$\gamma_0 \approx 2.2802369\;E52$.

These unique driving terms are inadmissible constellations until the counterexamples
themselves appear, in $\pgap(113^\#)$.  These $29$ counterexamples share a single common driving term up into
$\pgap(59^\#)$, at which point this inadmissible driving term has length $J=478$ and initial generator
\begin{eqnarray*}
\gamma_0(59^\#) &= & 107 + 6\cdot 11^\# + 8\cdot 13^\# + 9\cdot 17^\# + 5\cdot 19^\# + 7\cdot 23^\# + 1\cdot 29^\#  \\
 & & \quad + 23\cdot 31^\#  + 38\cdot 37^\# + 34\cdot 41^\# + 46\cdot 43^\# + 20\cdot 47^\# + 13 \cdot 53^\# \\
 & \approx & 4.36569294 \, E20
 \end{eqnarray*}

\begin{figure}[hbt]
\centering
\includegraphics[width=5in]{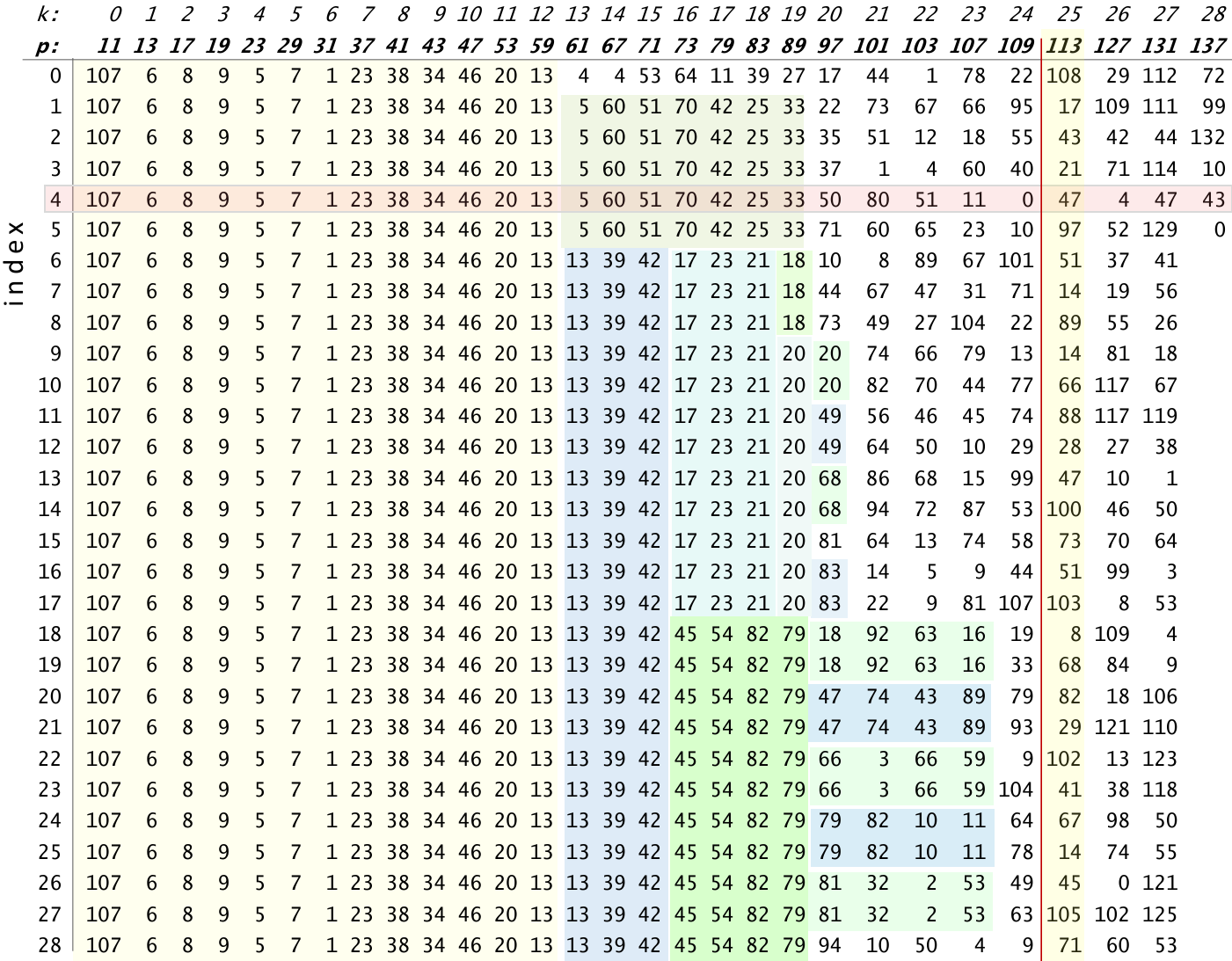}
\caption{\label{EngPrefixFig} Listed here are the primorial coordinates for the unique instances of each of the 29 
$(459,3242)$-counterexamples that have initial generator $\gamma_0=107$ in $\pgap(11^\#)$.  The driving terms
in $\pgap(p^\#)$ for $p < 113$ are inadmissible -- they do not survive.  The counterexamples themselves appear in $\pgap(113^\#)$.
The row of primorial coordinates for the prefix for our example $s_4$ is highlighted.}
\end{figure}

Beyond the unique prefixes listed in Figure~\ref{EngPrefixFig}, the instances of each counterexample branch out into a tree.
The number of nodes at first grows slowly but then explodes, as the product of the number of admissible instances $p-\nu_p$
at each stage of the sieve \cite{FBHktuple}.  Exhaustive breadth-first explorations of the instances of any of the
counterexamples run out very quickly.  

\begin{figure}[hbt]
\centering
\includegraphics[width=5in]{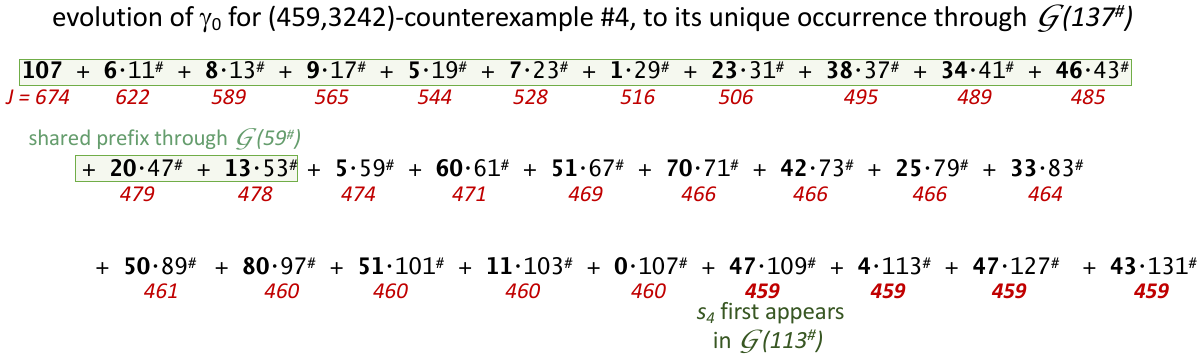}
\caption{ \label{Engprimm} The primorial coordinates for the unique occurrence of the $(459,3242)$-counterexample $s_4$, 
index $4$ in the list  in Figure~\ref{EngPrefixFig}.  The length of the unique occurrence of $s_4$ at each stage of the sieve is shown in red.  
In this evolution the inadmissible driving term in $\pgap(11^\#)$ undergoes interior fusions until it finally produces the admissible constellation 
$s_4$ in $\pgap(113^\#)$. After $\pgap(137^\#)$ there are multiple occurrences of this constellation in $\pgap(p^\#)$.}
\end{figure}

\subsection{Relative populations of $(459,3242)$-counterexamples.}
The $(459,3242)$-counterexamples have no admissible driving terms beyond the constellations themselves.  So their population model
across the cycles of gaps $\pgap(p^\#)$ is very simple.

\begin{lemma}
Let $s$ be one of the $58$ $(459,3242)$-counterexamples.  Then $s$ first appears in the cycle $\pgap(113^\#)$, and has a unique
instance into $\pgap(131^\#)$.  For $p \ge 137$ the population of $s$ in the cycle $\pgap(p^\#)$ is the product 
$$ n_{s,459}(p^\#) \; = \;  \prod_{137 \le q \le p} (q - \nu_q) $$
in which $q-\nu_q$ is the number of admissible residues $\bmod  \, q$ for the initial generator $\gamma_0$ for $s$.
\end{lemma}

\begin{proof}
The first parts of this lemma are computational results.
Figure~\ref{EngPrefixFig} shows the primorial coordinates for the $29$ counterexamples that have initial generator $\gamma_0=107$
in $\pgap(11^\#)$.  The other $29$ counterexamples are the reverse of these.  Figure~\ref{EngPrefixFig} provides the primorial coordinates
for the unique instances of $s$ in $\pgap(p^\#)$.

For primes $11 \le p < 113$ the constellations are represented by inadmissible driving terms, that is, by longer inadmissible constellations
which will produce $s$ under further interior fusions.  Figure~\ref{Engprimm} displays this evolution for the example $s_4$.

All $58$ of the $(459,3242)$-counterexamples first occur in $\pgap(113^\#)$.  Again see the summary data in Figure~\ref{EngPrefixFig}.  The
data and Python code are available in the GitHub repository.

Now the general population model from \cite{FBHktuple} becomes non-trivial.   For $p > 113$
$$n_{s,459}(p^\#) = \prod_{113 < q \le p} (q - \nu_q).$$
Since all of the $(459, 3242)$-counterexamples have unique images into $\pgap(131^\#)$, we can replace the bounds for that product
with $137 \le q \le p$.
\end{proof}

See Table~\ref{Eng4nadm} for the number of admissible instances $\bmod \, q$ for our
example $s_4$ and primes $11 \le q \le 499$.  
Similar data for the other $(459,3242)$-counterexamples are available in the GitHub repository.
We highlight the primes $q > 460$ (in blue) since these will fall within the second factor for the relative population model described below.

\begin{table}[htb]
\centering
\begin{tabular}{|rc||rc||rc||rc||rc||rc|}
\multicolumn{2}{c}{$p \hspace{0.2in} p-\nu_p$} & \multicolumn{2}{c}{$p \hspace{0.2in} p-\nu_p$} & \multicolumn{2}{c}{$p \hspace{0.2in} p-\nu_p$} & 
\multicolumn{2}{c}{$p \hspace{0.2in} p-\nu_p$} & \multicolumn{2}{c}{$p \hspace{0.2in} p-\nu_p$} & \multicolumn{2}{c}{$p \hspace{0.2in} p-\nu_p$} \\ \hline
{\em 13} & 1 & {\em 73} & 1 & 151 & 1 & 233 & 9 & 317 & 39 & 419 & 95 \\
{\em 17} & 1 & {\em 79} & 1 & 157 & 3 & 239 & 11 & 331 & 39 & 421 & 96 \\
{\em 19} &  1 & {\em 83} & 1 & 163 & 1 & 241 & 17 & 337 & 48 & 431 & 105 \\
{\em 23} &  1 & {\em 89} & 1 & 167 & 3 & 251 & 14 & 347 & 49 & 433 & 100 \\
{\em 29} &  1 & {\em 97} & 1 & 173 & 2 & 257 & 16 & 349 & 49 & 439 & 101 \\
{\em 31} & 1 & {\em 101} & 1 & 179 & 4 & 263 & 23 & 353 & 57 & 443 & 108 \\
{\em 37} & 1 & {\em 103} & 1 & 181 & 3 & 269 & 19 & 359 & 60 & 449 & 113 \\
{\em 41} & 1 & {\em 107} & 1 & 191 & 9 & 271 & 17 & 367 & 66 & 457 & 109 \\
{\em 43} & 1 & {\em 109} & 1 & 193 & 5 & 277 & 21 & 373 & 69 & \textcolor{blue}{461} & \textcolor{blue}{113} \\
{\em 47} & 1 & \textcolor{ForestGreen}{\bf 113} & \textcolor{ForestGreen}{1} & 197 & 5 & 281 & 28 & 379 & 66 & \textcolor{blue}{463} & \textcolor{blue}{131} \\
{\em 53} & 1 & 127 & 1 & 199 & 5 & 283 & 21 & 383 & 72 & \textcolor{blue}{467} & \textcolor{blue}{124} \\
{\em 59} & 1 & 131 & 1 & 211 & 11 & 293 & 30 & 389 & 83 & \textcolor{blue}{479} & \textcolor{blue}{134} \\
{\em 61} & 1 & 137 & 1 & 223 & 9 & 307 & 33 & 397 & 86 & \textcolor{blue}{487} & \textcolor{blue}{130} \\
{\em 67} & 1 & 139 & 2 & 227 & 13 & 311 & 37 & 401 & 82 & \textcolor{blue}{491} & \textcolor{blue}{141} \\
{\em 71} & 1 & 149 & 1 & 229 & 8 & 313 & 35 & 409 & 88 & \textcolor{blue}{499} & \textcolor{blue}{138} \\ \hline
\end{tabular}
 \caption{ \label{Eng4nadm}The number $p - \nu_p$ of admissible residues $\bmod p$ for the counterexample ${s_4 \in (459,3242)}$ over the 
 primes ${13 \le p \le 499}$.  The evolution begins with inadmissible driving terms in the cycles up through $\pgap(109^\#)$.
The constellation $s_4$ itself first appears in $\pgap(113^\#)$, and there continues to be a unique image of $s_4$ into $\pgap(137^\#)$.}
\end{table}



From \cite{FBHktuple} the asymptotic relative population of a constellation $s$ is the product of two factors.
The first factor is the product of admissible residues up through $k=J+1$.
\begin{equation}\label{Eqwinf}
 w_{s,J}(\infty) =  \prod_{q \le J+1} (q-\nu_q) \, \cdot \, \prod_{J+1 < q}\frac{q-\nu_q}{q-J-1} 
 \end{equation}

For the $(459,3242)$-counterexamples, $J=459$, and for $s_4$ this becomes
\begin{eqnarray*}
 w_{s_4,459}(\infty) & = & \prod_{q \le 460} (q-\nu_q) \, \cdot \, \prod_{460 < q \le 1621}\frac{q-\nu_q}{q-460}  \\
  & \approx & 1.907367 \, E72 \; \cdot \;  2.136352 \, E17 \; \approx \;  4.074808 \, E89
\end{eqnarray*}
Since $|s|=3242$, for every prime $q > 1621$ the factor $\frac{q-\nu_q}{q-460} = 1$.

\begin{table}[!h]
\centering
\begin{tabular}{rrrr|rrrr}
\# & $\prod_{q \le 460}$ & $\prod_{q > 460}$ & $w_s(\infty)$ & \# & $\prod_{q \le 460}$ & $\prod_{q > 460}$ & $w_s(\infty)$ \\ \hline
$0$ & $1.28E72$ & $2.47E17$ & $3.161E89$ &  $29$ & $1.16E73$ & $2.32E17$ & $2.691E90$ \\
$1$ & $3.96E72$ & $2.23E17$ & $8.810E89$ & $30$ & $2.10E73$ & $2.52E17$ & $5.293E90$ \\
$2$ & $4.22E72$ & $2.18E17$ & $9.203E89$ & $31$ & $1.76E73$ & $2.43E17$ & $4.281E90$ \\
$3$ & $2.60E72$ & $2.24E17$ & $5.805E89$ & $32$ & $3.48E73$ & $2.49E17$ & $8.657E90$ \\
$4$ & $1.91E72$ & $2.14E17$ & $4.075E89$ & $33$ & $2.92E73$ & $2.37E17$ & $6.922E90$ \\
$5$ & $1.46E72$ & $2.24E17$ & $3.281E89$ & $34$ & $2.55E73$ & $2.43E17$ & $6.188E90$ \\
$6$ & $6.58E72$ & $2.55E17$ & $1.674E90$ & $35$ & $3.04E73$ & $2.55E17$ & $7.741E90$ \\
$7$ & $3.54E72$ & $2.56E17$ & $9.058E89$ & $36$ & $2.59E73$ & $2.51E17$ & $6.508E90$ \\
$8$ & $7.73E72$ & $2.50E17$ & $1.935E90$ & $37$ & $2.11E73$ & $2.40E17$ & $5.063E90$ \\
$9$ & $3.99E72$ & $2.51E17$ & $1.004E90$ & $38$ & $1.21E73$ & $2.56E17$ & $3.085E90$ \\
$10$ & $4.58E72$ & $2.54E17$ & $1.161E90$ & $39$ & $1.00E73$ & $2.44E17$ & $2.443E90$ \\
$11$ & $9.12E72$ & $2.48E17$ & $2.263E90$ & $40$ & $8.19E72$ & $2.52E17$ & $2.067E90$ \\
$12$ & $9.76E72$ & $2.50E17$ & $2.441E90$ & $41$ & $7.16E72$ & $2.50E17$ & $1.791E90$ \\
$13$ & $1.01E73$ & $2.51E17$ & $2.536E90$ & $42$ & $1.30E73$ & $2.46E17$ & $3.190E90$ \\
$14$ & $1.16E73$ & $2.55E17$ & $2.950E90$ & $43$ & $1.16E73$ & $2.55E17$ & $2.950E90$ \\
$15$ & $1.30E73$ & $2.46E17$ & $3.190E90$ & $44$ & $1.01E73$ & $2.51E17$ & $2.536E90$ \\
$16$ & $7.16E72$ & $2.50E17$ & $1.791E90$ & $45$ & $9.76E72$ & $2.50E17$ & $2.441E90$ \\
$17$ & $8.19E72$ & $2.52E17$ & $2.067E90$ & $46$ & $9.12E72$ & $2.48E17$ & $2.263E90$ \\
$18$ & $1.00E73$ & $2.44E17$ & $2.443E90$ & $47$ & $4.58E72$ & $2.54E17$ & $1.161E90$ \\
$19$ & $1.21E73$ & $2.56E17$ & $3.085E90$ & $48$ & $3.99E72$ & $2.51E17$ & $1.004E90$ \\
$20$ & $2.11E73$ & $2.40E17$ & $5.063E90$ & $49$ & $7.73E72$ & $2.50E17$ & $1.935E90$ \\
$21$ & $2.59E73$ & $2.51E17$ & $6.508E90$ & $50$ & $3.54E72$ & $2.56E17$ & $9.058E89$ \\
$22$ & $3.04E73$ & $2.55E17$ & $7.741E90$ & $51$ & $6.58E72$ & $2.55E17$ & $1.674E90$ \\
$23$ & $2.55E73$ & $2.43E17$ & $6.188E90$ & $52$ & $1.46E72$ & $2.24E17$ & $3.281E89$ \\
$24$ & $2.92E73$ & $2.37E17$ & $6.922E90$ & $53$ & $1.91E72$ & $2.14E17$ & $4.075E89$ \\
$25$ & $3.48E73$ & $2.49E17$ & $8.657E90$ & $54$ & $2.60E72$ & $2.24E17$ & $5.805E89$ \\
$26$ & $1.76E73$ & $2.43E17$ & $4.281E90$ & $55$ & $4.22E72$ & $2.18E17$ & $9.203E89$ \\
$27$ & $2.10E73$ & $2.52E17$ & $5.293E90$ & $56$ & $3.96E72$ & $2.23E17$ & $8.810E89$ \\
$28$ & $1.16E73$ & $2.32E17$ & $2.691E90$ & $57$ & $1.28E72$ & $2.47E17$ & $3.161E89$ \\
\hline
\end{tabular}
\caption{ The asymptotic relative populations $w_s(\infty)$ for the $58$ Engelsma $(459,3242)$-counterexamples, and the
two factors of $w_s(\infty)$ from Equation~\ref{Eqwinf}.}
\end{table}

Note that $s_{25}$ and its reversal $s_{32}$ have asymptotic relative populations of $8.657E90$, $27.4$ times higher than
the asymptotic relative populations for $s_{0}$ and its reversal $s_{57}$.
Most of this variation occurs in the first factor, the number of instances for $q \le 460$.

\begin{figure}[hbt]
\centering
\includegraphics[width=5in]{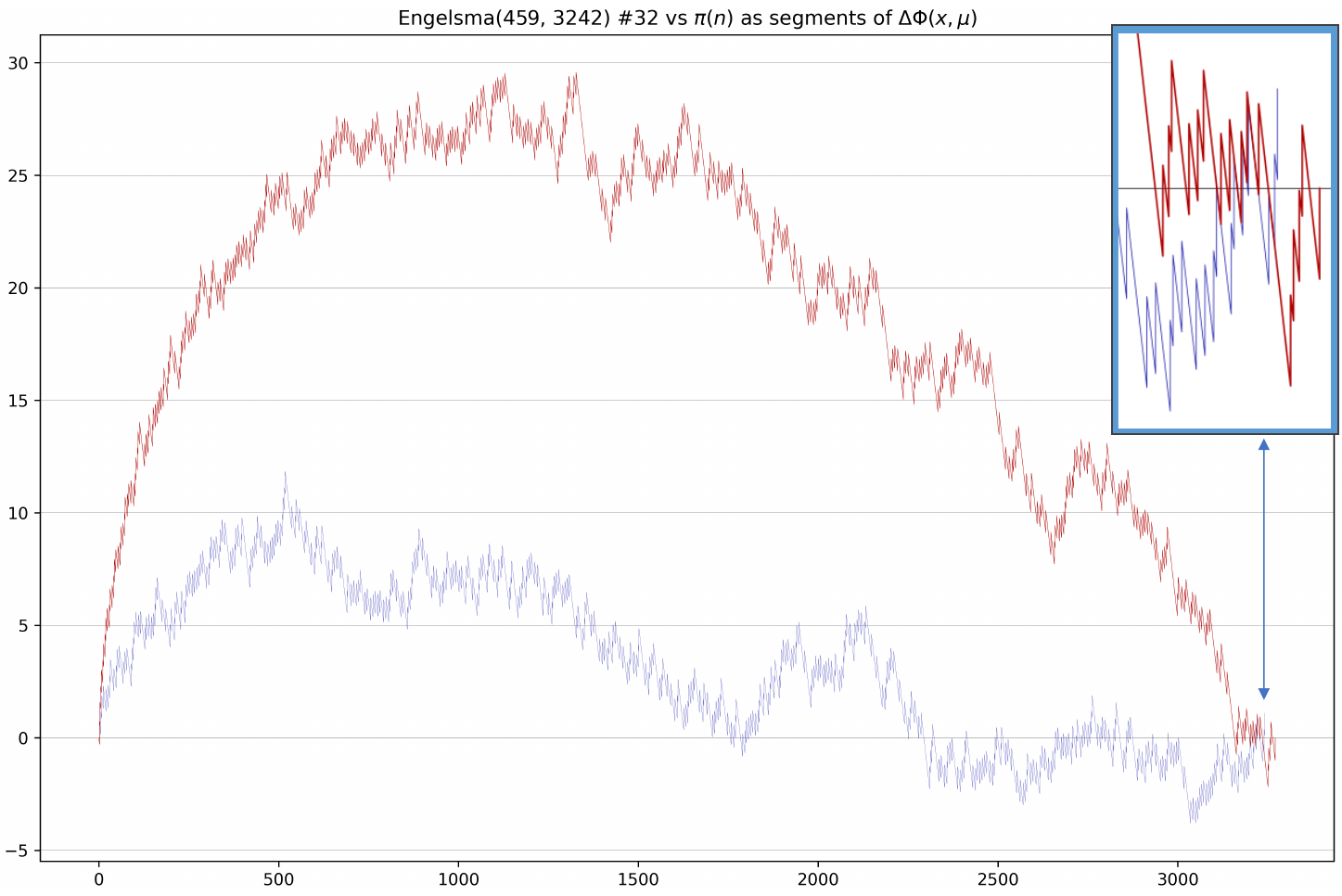}
\caption{\label{Eng32vpiFig} The $(459,3242)$-counterexample $s_{32}$ plotted (in blue) as a segment of 
$\Delta \Phi(x,\mu)$, compared to the
gaps between primes (in red) from $0$ to $x$. The small gaps in $\pi(x)$ create a huge early lead, but as larger gaps occur between primes,
it gives narrow constellations an opportunity to overtake it.  Again $s_{32}$ first crosses above the red graph for $(458,3240)$, and stays above for $(459,3242)$.}
\end{figure}

We graph the constellation $s_{32}$ against $\pi(x)$ in Figure~\ref{Eng32vpiFig}.  Not only does $s_{32}$ have the highest asymptotic
relative population among the $(459,3242)$-counterexamples, but the inset shows that as we reach the end of $s_{32}$, the truncated
constellation provides examples for $452 \le J \le 458$ in Engelsma's table, Figure~\ref{EngTblFig}, as well.  
We see the graph of $s_{32}$ jump onto the downward sloping segments
of $\pi(x)$.

\begin{table}
\begin{tabular}{r|cr|cr|r|}
 & \multicolumn{2}{c}{$\pi(x)$} & \multicolumn{2}{c}{$s_{32}$} & \\
 $J$ & $p_J$ & $g_J$ & $\gamma_J$ & $g_J$ & $J-\pi(\gamma_J)$ \\ \hline
 $450$ & $3181$ & $12$ & $3192$ & $6$ &  $-2$ \\
 $451$ & $3187$ & $6$ & $3198$ & $6$ & $-1$ \\
 $452$ & $3191$ & $4$ & $3200$ & $2$ & \textcolor{PineGreen}{$0$} \\
 $453$ & $3203$ & $12$ & $3210$ & $10$ & $-1$ \\
 $454$ & $3209$ & $6$ & $3212$ & $2$ & \textcolor{PineGreen}{$0$} \\
 $455$ & $3217$ & $8$ & $3218$ & $6$ & \textcolor{PineGreen}{$0$} \\
 $456$ & $3221$ & $4$ & $3222$ & $4$ & \textcolor{PineGreen}{$0$} \\
 $457$ & $3229$ & $8$ & $3236$ & $14$ & \textcolor{PineGreen}{$0$} \\
 $458$ & $3251$ & \textcolor{red}{$22$} & $3240$ & $4$ & \textcolor{NavyBlue}{$+1$} \\
 $459$ & $3253$ & $2$ & $3242$ & $2$ & \textcolor{NavyBlue}{$+2$} \\ \hline
\end{tabular}
\caption{\label{s32Tbl} The tail of $s_{32}$ provides an example of how narrow constellations identified by Engelsma for $452 \le J \le 459$
come to surpass $\pi(x)$.  The gap of $22$ from $p=3229$ to $3251$ finally provides the opportunity for $s_{32}$ to jump ahead.
See Figures~\ref{Eng32vpiFig} and~\ref{Eng4vpiFig}.}
\end{table}

With short constellations, say with $2 \le J \le 8$, we can easily identify a dozen reference constellations with which to compare
relative populations.  For the length $J=459$ we have very few meaningful reference constellations, to which to compare
the counterexamples.

One constellation that we do have at hand is a repetition of a gap.  
For a repetition of length $459$, the minimal gap ${g_{\rm rep}=457^\# \approx 2.19145\;E187}$.  A repetition of length $459$ of this gap
$g_{\rm rep}$ is a constellation of span $1.005874585188\;E190$ (vs $3242$ for the Engelsma counterexamples). 
The asymptotic relative population of this repetition is $1.992766199\;E186$ (vs $8.657E90$ for $s_{32}$).

Recall from our earlier work \cite{FBHPatterns, FBHktuple} that the relative populations can only be compared across admissible constellations 
of the same length $J$.

\subsection{Searching for surviving instances}
From the work above we know that all of the Engelsma $(459,3242)$-counterexamples persist across stages of Eratosthenes sieve, and
we know their asymptotic relative populations.  Where among the primes would we see an instance of any of them?

We undertook a breadth-first search across all $58$ counterexamples, exploring {\em all} admissible instances.  Looking at Table~\ref{Eng4nadm}
as one example, for $s_4$, we see that breadth-first searches quickly run into a computational barrier of superexponential growth.
The number of admissible instances of $s_j$ in the cycle $\pgap(p_k^\#)$ is
\begin{eqnarray*}
n_{s_j, 459}(p_k^\#) & = & \prod_{q \le p_k} (q - \nu_q(s)) \\
     & = & \prod_{137 \le q \le p_k} (q - \nu_q(s))
\end{eqnarray*}
Figure~\ref{EngPrefixFig} shows that each of the counterexamples $s_j$ has a unique occurrence up into $\pgap(131^\#)$.
For $s_4$, a breadth-first search over admissible instances becomes intractable long before we have even reached the transition
point $p_k > 460$.
 
\begin{table}
\begin{tabular}{|rcr||rcr||rcr|}
\multicolumn{1}{c}{$p$} & \multicolumn{1}{c}{$p-\nu_p$} & \multicolumn{1}{c}{$n_{s_4}(p^\#)$} & 
\multicolumn{1}{c}{$p$} & \multicolumn{1}{c}{$p-\nu_p$} & \multicolumn{1}{c}{$n_{s_4}(p^\#)$} & 
\multicolumn{1}{c}{$p$} & \multicolumn{1}{c}{$p-\nu_p$} & \multicolumn{1}{c}{$n_{s_4}(p^\#)$} \\ \hline
$137$ & $1$ & $1$ & $191$ & $9$ & $ 3888$ & $241$ & $17$ & $ 8.421\,E12$ \\
$139$ & $2$ & $2$ & $193$ & $5$ & $ 19440$ & $251$ & $14$ & $ 1.179\,E14$ \\
$149$ & $1$ & $2$ & $197$ & $5$ & $ 97200$ & $257$ & $16$ & $ 1.886\,E15$ \\
$151$ & $1$ & $2$ & $199$ & $5$ & $ 486000$ & $263$ & $23$ & $ 4.339\,E16$ \\
$157$ & $3$ & $6$ & $211$ & $11$ & $ 5346000$ & $269$ & $19$ & $ 8.244\,E17$ \\
$163$ & $1$ & $6$ & $223$ & $9$ & $ 4.811\,E7$ & $271$ & $17$ & $ 1.401\,E19$ \\
$167$ & $3$ & $18$ & $227$ & $13$ & $ 6.255\,E8$ & $277$ & $21$ & $ 2.943\,E20$ \\
$173$ & $2$ & $36$ & $229$ & $8$ & $ 5.004\,E9$ & $281$ & $28$ & $ 8.240\,E21$ \\
$179$ & $4$ & $144$ & $233$ & $9$ & $ 4.503\,E10$ & $283$ & $21$ & $ 1.730\,E23$ \\
$181$ & $3$ & $432$ & $239$ & $11$ & $ 4.954\,E11$ & $293$ & $30$ & $ 5.191\,E24$ \\ \hline
\end{tabular}
\caption{The number of admissible instances of the counterexample $s_4$ in the cycle $\pgap(p_k^\#)$, just after
$s_4$ has a unique instance in $\pgap(137^\#)$.}
\end{table}

\begin{table}
\begin{tabular}{rrr|rrr|rrr}
$\#$ & $m_k*p_{k-1}^\#$ & $\sim \gamma_0$ & $\#$ & $m_k*p_{k-1}^\#$ & $\sim \gamma_0$ & $\#$ & $m_k*p_{k-1}^\#$ & $\sim \gamma_0$ \\ \hline
$0$ & $\lil 2\cdot 193^\#$ & $\lil 5.330\,E77$ & $20$ & $\lil 60\cdot 191^\#$ & $\lil 6.250\,E76$ & $40$ & $\lil 1\cdot 193^\#$ & $\lil 3.588\,E77$ \\
$1$ & $\lil 6 \cdot 191^\#$ & $\lil 6.511\,E75$ & $21$ & $\lil 44 \cdot 191^\#$ & $\lil 4.617\,E76$ & $41$ & $\lil  2\cdot 193^\#$ & $\lil 5.380\,E77$ \\
$2$ & $\lil 7\cdot 191^\#$ & $\lil 7.960\,E75$ & $22$ & $\lil 80\cdot 181^\#$ & $\lil 4.335\,E74$ & $42$ & $\lil 11\cdot 191^\#$ & $\lil 1.186\,E76$ \\
$3$ & $\lil 87\cdot 191^\#$ & $\lil 8.973\,E76$ & $\bf 23$ & $\lil 18\cdot 181^\#$ & $\lil 9.733\,E73$ & $43$ & $\lil 58\cdot 191^\#$ & $\lil 5.984\,E76$ \\
$4$ & $\lil 6\cdot 193^\#$ & $\lil 1.331\,E78$ & $24$ & $\lil 1\cdot 193^\#$ & $\lil 2.494\,E77$ & $44$ & $\lil 28\cdot 191^\#$ & $\lil 2.959\,E76$ \\
$5$ & $\lil 5\cdot 193^\#$ & $\lil 1.005\,E78$ & $25$ & $\lil 64\cdot 191^\#$ & $\lil 6.684\,E76$ & $45$ & $\lil 24\cdot 191^\#$ & $\lil 2.546\,E76$ \\
$6$ & $\lil 6\cdot 191^\#$ & $\lil 7.043\,E75$ & $26$ & $\lil 32\cdot 191^\#$ & $\lil 3.342\,E76$ & $46$ & $\lil 56\cdot 191^\#$ & $\lil 5.791\,E76$ \\
$7$ & $\lil 68\cdot 191^\#$ & $\lil 7.085\,E76$ & $27$ & $\lil 7\cdot 191^\#$ & $\lil 8.009\,E75$ & $47$ & $\lil 1\cdot 193^\#$ & $\lil 2.397\,E77$ \\
$8$ & $\lil 1\cdot 193^\#$ & $\lil 3.078\,E77$ & $28$ & $\lil 98\cdot 191^\#$ & $\lil 1.013\,E77$ & $48$ & $\lil 3\cdot 193^\#$ & $\lil 6.891\,E77$ \\
$9$ & $\lil 62\cdot 191^\#$ & $\lil 6.424\,E76$ & $29$ & $\lil 157\cdot 191^\#$ & $\lil 1.620\,E77$ & $49$ & $\lil 24\cdot 191^\#$ & $\lil 2.543\,E76$ \\
$10$ & $\lil 2\cdot 193^\#$ & $\lil 5.574\,E77$ & $30$ & $\lil 19\cdot 191^\#$ & $\lil 1.989\,E76$ & $50$ & $\lil 1\cdot 193^\#$ & $\lil 3.018\,E77$ \\
$11$ & $\lil 77\cdot 191^\#$ & $\lil 7.966\,E76$ & $31$ & $\lil 54\cdot 191^\#$ & $\lil 5.644\,E76$ & $51$ & $\lil 85\cdot 191^\#$ & $\lil 8.857\,E76$ \\
$12$ & $\lil 126\cdot 191^\#$ & $\lil 1.301\,E77$ & $32$ & $\lil 7\cdot 191^\#$ & $\lil 7.755\,E75$ & $52$ & $\lil 1\cdot 193^\#$ & $\lil 2.029\,E77$ \\
$13$ & $\lil 114\cdot 191^\#$ & $\lil 1.176\,E77$ & $33$ & $\lil 21\cdot 191^\#$ & $\lil 2.242\,E76$ & $53$ & $\lil 1\cdot 193^\#$ & $\lil 3.444\,E77$ \\
$14$ & $\lil 30\cdot 191^\#$ & $\lil 3.124\,E76$ & $34$ & $\lil 1\cdot 191^\#$ & $\lil 1.424\,E75$ & $54$ & $\lil 2\cdot 193^\#$ & $\lil 5.008\,E77$ \\
$15$ & $\lil 37\cdot 191^\#$ & $\lil 3.912\,E76$ & $35$ & $\lil 82\cdot 181^\#$ & $\lil 4.428\,E74$ & $55$ & $\lil 101\cdot 191^\#$ & $\lil 1.041\,E77$ \\
$16$ & $\lil 97\cdot 191^\#$ & $\lil 1.004\,E77$ & $36$ & $\lil 29\cdot 191^\#$ & $\lil 3.012\,E76$ & $56$ & $\lil 149\cdot 191^\#$ & $\lil 1.545\,E77$ \\
$17$ & $\lil 61\cdot 191^\#$ & $\lil 6.312\,E76$ & $37$ & $\lil 5\cdot 191^\#$ & $\lil 5.807\,E75$ & $57$ & $\lil 110\cdot 191^\#$ & $\lil 1.136\,E77$ \\
$18$ & $\lil 167\cdot 191^\#$ & $\lil 1.731\,E77$ & $38$ & $\lil 43\cdot 191^\#$ & $\lil 4.535\,E76$ &  & & \\
$19$ & $\lil 8\cdot 191^\#$ & $\lil 8.483\,E75$ & $39$ & $\lil 13\cdot 191^\#$ & $\lil 1.351\,E76$ & & & \\ \hline
\end{tabular}
\caption{ \label{GammaTbl} Listed is the smallest initial generator for an admissible instance for each of the $58$ $(459,3242)$-counterexamples, from our breadth-first search. None of these are guaranteed to survive further stages of the sieve, but any surviving instance would have $\gamma_0$
at least this large.}
\end{table}

The following lemma characterizes surviving instances of admissible constellations.

\begin{lemma}
Let $s$ be an admissible constellation of length $J$.  In primorial coordinates, for each surviving instance of $s$ there is a $k$ 
such that the initial generator 
$$
\gamma_0(p_k) = \gamma_0 + m_1 \cdot p_0^\# + m_2 \cdot p_1^\# + \cdots + m_k p_{k-1}^\# 
$$
with $0\le m_j < p_j$ for $1\le j < k$ and $0 < m_k < p_k$, after which $m_j = 0$ for all $j > k$ until
$\gamma_0(p_k)$ falls within an horizon of survival $H(q)$ for some prime $q$.
\end{lemma}

\begin{proof}
In order to survive the sieve and be confirmed as a constellation among primes, the instance of $s$ must pass eventually into a
horizon of survival $H(q_0)=(q_0^2, q_1^2)$.
In the replication step R2 of the recursion \cite{FBHPatterns, FBHktuple}, at some point, say $p_{k+1}$, the instance must remain
in the first copy of $\pgap(	p_j^\#)$ across the last several stages of the sieve.  This requires $m_j=0$ for $j > k$, and $\gamma_0(p_k)$
to remain as an admissible instance of $s$ in $\pgap(p_j^\#)$.

Let $k$ be the last index for which $m_k > 0$.  We include here the degenerate case in which $\gamma_0$ is the last positive
coefficient and $m_j = 0$ for all $j > 0$.  Then $\gamma_0(p_k)$ is an admissible instance of $s$ among primes.
\end{proof}

As we switch from exhaustive breadth-first searches to opportunistic depth-first searches, we are looking for instances of $s$
that have long runs of $m_j=0$ as admissible primorial coefficients at the end of the search.  
As soon as an $m_j>0$ the instance has jumped a distance of $m_j \cdot p_{j-1}^\#$, 
beyond the first span of $\pgap(p_{j-1}^\#)$.

The probability of finding a surviving instance of a $(459,3242)$-counterexample is quite small.  There are unique instances of
all $58$ of these counterexamples in $\pgap(131^\#)$.  Of the $5.259\,E50$ gaps in this cycle, only $58$ start one of these counterexamples.
We have seen above that no instances survive before $9.733\,E73$.

We may identify an occasional surviving instance of one of these $(459,3242)$-counterexamples, but we do not expect their density
in quadratic intervals $[n^2,(n+1)^2]$ to steadily increase until the average gap size $\mu$ satisfies
$$ \mu > 460$$
This rough estimate is derived from Lemma 3.2 of \cite{FBHquad}.  By Merten's Third Theorem we estimate that $\mu > 460$ when
$p > 1.465\,E112$.


\section{Conclusion}
We have extended our studies of the $k$-tuple conjecture \cite{FBHktuple, FBHPatterns}, applying those methods to the 
$58$ possible counterexamples to the convexity conjecture of length $J=459$ and span $|s|=3242$, identified by Engelsma.

We show the evolution of each of these $58$ counterexamples from inadmissible driving terms in the cycle $\pgap(11^\#)$,
through their first appearance in $\pgap(113^\#)$ and continue an exhaustive breadth-first search up through $\pgap(211^\#)$.
We record the asymptotic relative population $w_s(\infty)$ for each one, and note the smallest instance (initial generator) $\gamma_0$
for each.  No instance of any of the $58$ $(459,3242)$-counterexamples occur among primes before $9.7\,E73$.

\bibliographystyle{alpha}
\bibliography{primes2024}

\end{document}